\documentclass[10pt,reqno]{amsart}
\usepackage{amsmath,amsthm,amsfonts,amssymb,amscd,amstext}
\usepackage[ansinew]{inputenc}
\usepackage[dvips]{graphicx}
\usepackage{psfrag,euscript}
\usepackage{pstricks}
\usepackage{pxfonts}
\newtheorem{maintheorem}{Theorem}

\newtheorem{theorem}{Theorem}[section]

\newtheorem{lemma}[theorem]{Lemma}

\newtheorem{remark}[theorem]{Remark}
\newtheorem{question}{Question}
\newtheorem{definition}{Definition}
\newtheorem{example}{Example}
\newtheorem{conjecture}{Conjecture}

\numberwithin{equation}{section}
\theoremstyle{plain}
\theoremstyle{definition}
\newcommand{\RR}{{\mathbb R}}

\newcommand{\NN}{{\mathbb N}}
\newcommand{\ZZ}{{\mathbb Z}}

\renewcommand{\epsilon}{\varepsilon}
\def \cR {{\mathcal R}}

\def \NN {{\mathbb N}}

\def \RR {{\mathbb R}}

\def \ZZ {{\mathbb Z}}

\title [Dense area-preserving homeomorphisms have zero Lyapunov exponents]
{Dense area-preserving homeomorphisms have zero Lyapunov exponents}
\author{M\'ario Bessa}
\email{bessa@fc.up.pt }
\address{Departamento de Matem\'atica da Universidade do Porto
\\
Rua do Campo Alegre 687,
4169-007 Porto, Portugal.\text{ and } ESTGOH-Instituto Polit\'ecnico de Coimbra
\\
Rua General Santos Costa,
3400-124 Oliveira do Hospital, Portugal.}

\author{C\'esar M. Silva}
\email{csilva@ubi.pt}
\address{
Universidade da Beira Interior
\\
Rua Marqu\^es d'\'Avila e Bolama,
  6201-001 Covilh\~a
Portugal.}
\date{\today}
\begin{document}
%----------------------------------------------------------------------------%
%----------------------------------------------------------------------------%
\begin{abstract}
We give a new definition (different from the one in~\cite{BS}) for a Lyapunov exponent (called \emph{new} Lyapunov exponent) associated to a continuous map. Our first result states that these new exponents coincide with the usual Lyapunov exponents if the map is differentiable. Then, we apply this concept to prove that there exists a $C^0$-dense subset of the set of the area-preserving homeomorphisms defined in a compact, connected and boundaryless surface such that any element inside this residual subset has zero \emph{new} Lyapunov exponents for Lebesgue almost every point. Finally, we prove that the function that associates an area-preserving homeomorphism, equipped with the $C^0$-topology, to the integral (with respect to area) of its top \emph{new} Lyapunov exponent over the whole surface cannot be upper-semicontinuous.
\end{abstract}
%----------------------------------------------------------------------------%
%----------------------------------------------------------------------------%
\maketitle


\bigskip

\noindent
\textbf{\textup{2000} Mathematics Subject Classification:}
Primary: 54H20, 28D05 Secondary:
37B99.

%\tableofcontents

%----------------------------------------------------------------------------%
%----------------------------------------------------------------------------%
\section{Introduction} \label{sec:introd}
%----------------------------------------------------------------------------%
%----------------------------------------------------------------------------%

%----------------------------------------------------------------------------%
\subsection{Ergodicity and smoothness}
%----------------------------------------------------------------------------%

Let $S$ be a compact, connected, boundaryless smooth Riemannian surface. Consider an area-form on $S$ and let $\lambdaup$ be the measure induced by it. We call $\lambdaup$ the \emph{Lebesgue measure} on $S$.  Let $\text{Hom\,}_{\lambdaup}(S)$ denote the space of homeomorphisms on $S$ which are $\lambdaup$-invariant, that is, $\lambdaup(h(B))=\lambdaup(B)$ for any Borelian set $B$. The space $\text{Hom\,}_{\lambdaup}(S)$ is metrizable with the metric
$$d(f, g) := \sup\{d(f(x), g(x)),d(f^{-1}(x), g^{-1}(x)) \colon x \in S\},$$
where $f$ and $g$ are in $\text{Hom\,}_{\lambdaup}(S)$, and this metric induces in $\text{Hom\,}_{\lambdaup}(S)$ a topology called the $C^0$ topology. We endow $\text{Hom\,}_{\lambdaup}(S)$ with the $C^0$ topology obtaining the Baire space $(\text{Hom\,}_{\lambdaup}(S),C^0)$. A property is said to be \emph{generic} if it holds in a residual subset, that is, the subset of points in which the property is valid contains a dense $G_{\delta}$ and, in particular, by Baire's Category theorem (see~\cite{Ku}) is $C^0$-dense in $\text{Hom\,}_{\lambdaup}(S)$.

We recall that $h\in \text{Hom\,}_{\lambdaup}(S)$ is said to be \emph{ergodic} if any $h$-invariant set has zero or full Lebesgue measure. In the early 1940's Oxtoby and Ulam ~\cite{OU} proved the following remarkable result which is stated for two-dimensional manifolds but is also true for higher dimensions.

\begin{theorem}\label{OU}
(Oxtoby-Ulam) Ergodicity holds in a residual subset of $\text{Hom\,}_{\lambdaup}(S)$.
\end{theorem}

We refer to Alpern and Prasad's work ~\cite{AP,AP2} and the references wherein for a historical background.

Since Poincar\'e, area-preserving discrete dynamical systems in surfaces have been often used in Classical and Celestial Mechanics, see e.g.~\cite[Appendix 9]{Ar78}. If we assume that the dynamical system is of class $C^r$ ($r\geq 3$), then the picture given by Theorem~\ref{OU} changes radically. In fact, by KAM theorem (see e.g. ~\cite{Y}), we obtain the prevalence of dynamically invariant circles supporting irrational rotations. Since these invariant circles have positive measure we conclude that being ergodic is not an open condition.

In this paper we are interested in understanding the dynamics for \emph{most} surface homemorphisms which preserve the Lebesgue measure.

\begin{subsection}{Dissipative \emph{versus} conservative dynamics}

For the larger class of \emph{dissipative} homeomorphisms in surfaces we have, for example, the following generic results:
\begin{enumerate}
\item [(A)] There exists a residual subset such that the closure of periodic points is equal to the non-wandering points (see~\cite{H});
\item [(B)] The property that $f$ has infinitely many periodic points of some finite period holds in a residual subset (see~\cite{PPSS});
\item [(C)] Generic homeomorphisms have some iteration with infinite topological entropy (see~\cite{K}).
\end{enumerate}

See ~\cite{AHK} for a very complete survey on the dissipative setting. We observe that some of the previous results are  available in the literature and concerning the area-preserving case, e.g., (A) was proved in ~\cite[Proposition 4]{DF}. Actually, in ~\cite{AD,DF} it was proved that chaotic homeomorphisms are abundant in the conservative setting. Here \emph{chaotic} is in the sense of Devaney (see~\cite{BBCDS}).

Since the lack of differentiability is a $C^0$-generic property it is natural to believe that the dynamics of homeomorphisms is very different from the dynamics of diffeomorphisms. For example, statement (B) above says that one of the consequences of the Kupka-Smale Theorem (see \cite{R} for the volume-preserving case) about finiteness of periodic points of a given period, for $C^r$ ($r\geq 1$), no longer holds in the  $C^0$-topology. In fact, hyperbolicity can be destroyed if a  small $C^0$-perturbation is allowed.
\end{subsection}

\begin{subsection}{\emph{New} Lyapunov exponents and statement of the results}

The stable/unstable manifold theory for homeomorphisms is very difficult to obtain and, as far as we know, holds in very special cases (see~\cite{BSl}). One standard way to obtain stable/unstable manifolds is by assuming non-zero Lyapunov exponents and then using Pesin's theory (\cite{BP}), which requires at least H\"{o}lder regularity\footnote{ It is worth to point out that, by recent results of Abdenur, Bonatti and Crovisier, (see~\cite{ABC}), for any $C^{1}$ diffeomorphism, with non-zero Lyapunov exponents and with a dominated splitting (see ~\cite{BDV} for the definition) with an index compatible with the non-uniformly hyperbolic splitting, we obtain a rich information about geometric
properties of the system, namely stable/unstable manifold theory. } of the tangent map. These assumptions imply some of the most important results in the mo\-dern theory of dynamical systems (\cite{BP}). We recall that the Lyapunov exponents measure the asymptotic behavior of the tangent map of a given dynamical system. A new approach and some additional care is needed if one want to define Lyapunov exponents for homeomorphisms.

In ~\cite{BS}, the second author and Barreira were able to define Lyapunov exponents adapted to the continuous case. In \S \ref{NLE} we develop another concept of Lyapunov exponents for continuous transformations, that we call \emph{new} Lyapunov exponents, and we will prove:

\begin{maintheorem}
  \label{main0}
If $f$ is a differentiable map then, for all $x \in S$ and $v \in T_x S$, we have that the Lyapunov exponents coincide with the \emph{new} Lyapunov exponents.
\end{maintheorem}

Once we have a good definition of Lyapunov exponents it is tempting to use non-zero Lyapunov exponents for area-pre\-ser\-ving homeomorphisms in order to obtain some kind of stable/unstable manifold. However, our next result says that, unfortunately, in $\text{Hom\,}_\lambdaup(S)$, we must be very careful since this property is rare. Actually, having non-zero Lyapunov exponents is not a $C^0$-open property and this is the content of Theorem~\ref{main}.

In order to state Theorem~\ref{main} let us consider the following function

\begin{equation}\label{entropy}
\begin{array}{cccc}
\Lambda\colon & \text{Hom\,}_\lambdaup (S) & \longrightarrow & [0,+\infty[ \\
& h & \longmapsto & \int_S \chi_N^{+}(h,x)d\lambdaup(x),
\end{array}
\end{equation}
where $\text{Hom\,}_\lambdaup (S)$ is endowed with the $C^0$-topology and $\chi_N^{+}$ is the top \emph{new} Lyapunov exponent (to be defined in \S\ref{NLE}). We say that $h\in \text{Hom\,}_\lambdaup (S)$ has zero top \emph{new} Lyapunov exponents for $\lambdaup$-a.e. point in $S$, if $\Lambda(h)=0$.

\begin{maintheorem}
  \label{main}
  Given any $h\in \text{Hom\,}_\lambdaup(S)$ and $\epsilon>0$, there exists $g\in \text{Hom\,}_\lambdaup(S)$,  $\epsilon$-$C^0$-close to $h$, such that $\Lambda(g)=0$.
  \end{maintheorem}

We mention that Bochi \cite{M2,Boc} proved that $C^1$-generically we have the dichotomy: zero Lyapunov exponents \emph{versus} hyperbolicity. Furthermore, in a remarkable work, Viana (see~\cite{V}), proved that for the $C^{\alpha}$ ($\alpha>0$) setting of discrete-time cocycles the hyperbolicity prevails. Actually, for the $C^{1+\alpha}$-diffeomorphisms case (which corresponds to the $C^{\alpha}$ tangent dynamical cocycle) there is no proof yet. On the other hand, in~\cite{AB}, using a result by Arnold and Cong (see  \cite{AC}) it is proved that for weaker topologies generic conservative cocycles have zero Lyapunov exponents which somehow is related to our result.

\medskip

\begin{center}
\small{\begin{tabular}{|c|c|c|c|r|}
	\hline
 weak topologies ($C^0$, $L^p$) & $C^1$-topology   &   strong topologies ($C^{1+\alpha}$)  \\
	\hline
{\tiny Theorem~\ref{main} (dense), Arbieto-Bochi (\cite{AB}) (generic)}   & {\tiny Bochi-Ma\~n\'e-Viana (\cite{M2,Boc,BV2})} & {\tiny Viana (\cite{V})} \\
	\hline
{\tiny One-point spectrum (o.p.s.)}   & {\tiny o.p.s. versus hyperbolicity} & {\tiny hyperbolicity} \\
	\hline
\end{tabular}}
\end{center}

\medskip

Our proof of Theorem~\ref{main} is supported in~\cite{Boc} and in a \emph{smoothing} result of area-preserving homeomorphisms in low dimension\footnote{\, We would like to thank Carlos Matheus for pointing us the smoothing result of Yong-Geun Oh.} (see~\cite{Oh,S}). Notice that these results only work in dimension two and three. For that recall that, due to Munkres theorem ~\cite{Mu}, in dimension $n\geq 3$ there are (dissipative) homeomorphisms which are not $C^0$-approximated by diffeomorphisms and this is a crucial step to obtain the aforementioned smoothing statement.

As a simple example of ergodic surface maps with zero Lyapunov exponents we have the uniquely ergodic rationally independent translation on the torus. As a consequence of Theorem~\ref{OU} and of Theorem~\ref{main} we obtain that maps with one of these properties (or eventually both properties) are abundant in surfaces, at least, in the $C^0$ sense. If we had a version of Theorem~\ref{main} for \emph{residual} instead of \emph{dense}, then combining Theorem~\ref{main} with Theorem~\ref{OU} we should obtain a generic version of Theorem~\ref{main}.

\begin{conjecture}\label{conjecture}
There exists a residual set $\cR\subset \text{Hom\,}_\lambdaup(S)$ such that, for every $f\in\cR$, $f$ is ergodic and, moreover, Lebesgue almost every point in $S$ has zero \emph{new} Lyapunov exponents.
\end{conjecture}

The crucial step in order to obtain the residual part in proof Bochi-Man\'e's dichotomy~\cite{B} is to use that the function that associates to any $f$ in the set of area-preserving $C^1$-diffeomorphisms, equipped with the $C^1$-topology, the number $$\hat{\Lambda}(f)=\int_S \chi^+(f,x) d\lambdaup(x),$$
where $\chi^+$ represents the top (or upper) Lyapunov exponent, is upper-semicontinuous. So, we could prove Conjecture~\ref{conjecture} once we get the same type of behavior for the function $\Lambda(h)$ defined in (\ref{entropy}). In \S 3 we will prove that such a function cannot be upper-semicontinuous, which is the content of Theorem~\ref{main2}.

%----------------------------------------------------------------------------%
\subsection{Generalizations of Theorem~\ref{main}}\label{G}
%----------------------------------------------------------------------------%

The three-dimensional version of Theorem~\ref{main} should follow from the strategy used in~\cite{BV2} and also from easier   topological arguments where the key objects are topological realizable sequences (see the definition in \S~\ref{TRS}). Observe that dominated splitting (see definition in \cite{BV2}) is not an obstruction to blending different expansion rates as it was with the $C^1$-topology.

So let $M$ be a compact, connected, boundaryless, smooth three-dimensional manifold and let $\lambdaup$ be the Lebesgue measure on $M$.

\begin{conjecture}
  \label{mainB}
  Given any $h\in \text{Hom\,}_\lambdaup(M)$ and $\epsilon>0$, there exists $g\in \text{Hom\,}_\lambdaup(M)$,  $\epsilon$-$C^0$-close to $h$, such that all its \emph{new} Lyapunov exponents are zero.
\end{conjecture}

We mention that, due to a result of M\"{u}ller~\cite{Mul}, we also have a smoothing version for homeomorphisms in dimension $n\geq 5$. Thus,  Conjecture~\ref{mainB} can be stated for manifolds $M$ such that $\dim(M)\geq 5$. On the other hand we remark that we have no idea how to prove this result when $\dim(M)= 4$.

\end{subsection}

%----------------------------------------------------------------------------%
\section{Lyapunov exponents for homeomorphisms}\label{NLE}
%----------------------------------------------------------------------------%

The aim of this section is to introduce \emph{new} Lyapunov exponents for continuous maps. In addition, we will prove that for differentiable maps the new exponents coincide with the classical ones (Theorem~\ref{main0}).

Our definition of the Lyapunov exponents is based on the definition given in~\cite{BS}, although in that paper the construction is done only in $\RR^n$ and the authors only define $n$ values associated with the Lyapunov exponent at each point. We should also remark that the order of the limits in~\eqref{eq:LyapunovExp} is the opposite of the order considered in~\cite{BS} and that allow us to show that for differentiable maps our Lyapunov exponents always coincide with the classical ones. It should also be mentioned that in~\cite{BS} the $n$ values of the Lyapunov exponent for non differentiable maps are defined with the additional purpose of establishing a lower bound for the Hausdorff dimension of some geometric constructions in $\RR^n$.
Previous attempts to use Lyapunov exponents in the non differentiable setting were done in~\cite{YK, B}. Namely, Kifer introduces in~\cite{YK} the maximal and minimal values of the Lyapunov exponent for arbitrary continuous maps in metric spaces and in~\cite{B} Barreira introduced different numbers that replace the same two values in the case of repellers of maps that are not necessarily differentiable and used them to obtain estimates for the Hausdorff dimension of the referred repellers. We stress that to the best of our knowledge the present work is the first where, given an arbitrary continuous map $f\colon M \to M$ and $x \in M$, a value for the Lyapunov exponent is defined for each $v \in T_x M$.

Like in all the other works mentioned, the approach in our definition of the new exponents will be to imitate the derivative as much as possible. Unfortunately, since our ``substitute'' for the derivative is not a
cocycle for an arbitrary continuous map, we are not able to show that the \emph{new} Lyapunov exponent is in fact a Lyapunov exponent (in the sense of the abstract theory defined in~\cite{BVGN}) and thus desirable properties like the invariance of the exponent along the orbit of a point must be derived directly from the definition. Under some integrability assumptions, an application of Kingmann ergodic theorem allow us to conclude that the top new exponent is invariant in the orbit of almost all points.

%----------------------------------------------------------------------------%
\subsection{\emph{New} Lyapunov exponents and proof of Theorem~\ref{main0}}
%----------------------------------------------------------------------------%

We still denote by $S$ a compact, connected, boundaryless smooth Riemannian surface and by $\lambdaup$ the Lebesgue measure on $S$. By compactness of $S$ and Darboux's theorem (see, e.g.~\cite{Ar78}), there exists a finite atlas $\mathcal{A}=\{\varphi_i\colon U_i\rightarrow\mathbb{R}^2;\,\, i=1,...,k\}$, where $U_i\subset S$ is an open set, and the charts are conservative. We assume  that, for any $x\in S$, we choose univocally $i(x):=\min \{i\in 1,...,k\colon x\in U_i\}$. The Riemannian metric fixed in the beginning shall not be used, instead we will consider the metric in $T_x S\ni v$ defined by $\|v\|:= \|D\varphi_{i(x)}\cdot v\|$, thus our computations in the sequel will be performed in the Euclidean space via the fixed charts.

Let $f\colon S \to S$ be a continuous map and, for each $x \in M$, $\delta > 0$ and $n \in \NN$, consider the set
$$B_x(\delta, n):=\{y \in M \colon d(f^j(x), f^j(y))<\delta \ \ {\rm for } \ \ j=0, 1, \dots, n\}$$
that we call \emph{dynamical ball}. For each $x \in S$ and $v \in T_x S$ define the \emph{new Lyapunov exponent of $f$ at $x$} by
\begin{equation} \label{eq:LyapunovExp}
\chi^+_N (f,x,v) :=  \limsup_{n \to +\infty} \lim_{\delta \to 0} \frac{1}{n} \log \sup_{y \in (B_x(\delta,n) \setminus \{x\}) \cap L_{x,v}}
\Delta(f,n,x,y)
\end{equation}
where
$$
\Delta(f,n,x,y) := \frac{\|f^n(x)-f^n(y)\|}{\|x-y\|} \,\,\,\text{ and }\,\,\, L_{x,v}:=\{x+kv \in \RR^2: k \in \mathbb{R} \}.
$$
Additionally, for each $x \in S$, we define the number
\begin{equation} \label{eq:Numero++}
\chi^{+}_N (f,x) :=  \limsup_{n \to +\infty} \lim_{\delta \to 0} \frac{1}{n} \log \sup_{y \in B_x(\delta,n) \setminus \{x\}} \Delta(f,n,x,y)
\end{equation}
and we call it the \emph{top new Lyapunov exponent}.

Clearly, for all $v \in T_x S$, we have
\[
\chi^+_N (f,x,v) \le \chi^{+}_N (f,x).
\]

\begin{remark}
In~\eqref{eq:LyapunovExp} and~\eqref{eq:Numero++}, though we do not write it explicitly, by $B_x(\delta,n) \setminus \{x\}$ we mean the set
$\varphi_{i(x)} \left( B_x(\delta,n) \setminus \{x\} \right)$ (note that we can assume that $\delta>0$ is sufficiently small and therefore assume that $B_x(\delta,n) \setminus \{x\}$
is contained in $U_{i(x)}$). Naturally, we also identify the vector $v \in T_x M$ with the corresponding $D\varphi_{i(x)}\cdot v \in \RR^2$ and that allow us to define $L_{x,v}$. We abuse the notation and keep denoting by $v$ instead of $D\varphi_{i(x)}\cdot v$ in order to avoid heavy notation.
\end{remark}

If $f$ is a differentiable map, we denote by $\chi^+$ the Lyapunov exponent defined in the standard way:
\[
\chi^+(f,x,v) := \limsup_{n \to \infty} \frac{1}{n} \log \|Df^n_x\cdot v\|
\]
and we also consider the numbers
\[
\chi^{+}(f,x) := \limsup_{n \to \infty} \frac{1}{n} \log \ \sup_{v \ne 0 } \frac{\|Df^n_x\cdot v\|}{\|v\|}.
\]

\begin{remark}
It follows from~\cite[Theorem 3]{BS} that if $f$ is a differentiable map of a compact manifold and $\mu$ is a finite $f$-invariant measure then
for $\mu$-almost every $x$ we have

$$ \chi^{+}(f,x)= \sup_{v \in T_xM} \chi^+(f,x,v).$$
\end{remark}

Theorem~\ref{main0} is included in the following result.

\begin{theorem}
If $f$ is a differentiable map we have $\chi_N^+(f,x,v)=\chi^+(f,x,v)$
 for all $x \in S$ and $v \in T_xS$ and $\chi_N^{+}(f,x)=\chi^{+}(f,x)$ for all $x \in S$.
\end{theorem}

\begin{proof} We will prove that, if $f$ is a differentiable map then, for all $x \in S$ and $v \in T_x S$, we have
$$\chi^+_N(f,x,v)=\chi^+(f,x,v).$$

Let $x \in S$ and $v \in T_x S$. Since $f$ is differentiable, for all
$n \in \NN$ there is a function $r_n$ such that we have
\[
\frac{\|f^n(x)-f^n(y)\|}{\|x-y\|}=\frac{\|Df^n(y-x)+r_n(y-x)\|}{\|y-x\|}
\]
and $r_n(\varepsilon) \to 0$ as $\varepsilon \to 0$. Therefore
\[
\begin{split}
& \sup_{y \in B_x(\delta,n) \setminus \{x\}\cap L_{x,v}} \Delta(f,n,x,y) \\
\le & \sup_{y \in B_x(\delta,n) \setminus \{x\}\cap L_{x,v}} \frac{\|Df^n(y-x)\|}{\|y-x\|}+
\sup_{y \in B_x(\delta,n) \setminus \{x\}\cap L_{x,v}} \frac{\|r_n(y-x)\|}{\|y-x\|}\\
\le & \|Df^n_x |_{L_{x,v}}\|+\phi_{n,x}(\delta),
\end{split}
\]
where
\[
\phi_{n,x}(\delta):=\sup_{y \in B_x(\delta,n) \setminus \{x\} \cap L_{x,v}} \frac{\|r_n(y-x)\|}{\|y-x\|}.
\]
Since for each fixed $n \in \NN$, $ \frac{\|r_n(y-x)\|}{\|y-x\|} \to 0$ and the map $\delta \mapsto \phi_{n,x}(\delta)$ is decreasing we have
$\displaystyle \lim_{\delta \to 0} \phi_{n,x}(\delta)=0$ and thus
\[
\begin{split}
& \lim_{\delta \to 0} \frac{1}{n} \log \sup_{y \in B_x(\delta,n)\setminus \{x\}\cap L_{x,v}} \Delta(f,n,x,y) \\
\le & \frac{1}{n} \log  \left[ \lim_{\delta \to 0}  \sup_{y \in B_x(\delta,n)\setminus \{x\} \cap L_{x,v}} \frac{\|Df^n(y-x)\|}{\|y-x\|}
+ \lim_{\delta \to 0} \phi_{n,x}(\delta) \right]\\
= &  \frac{1}{n} \log \|Df^n_x |_{L_{x,v}}\|.
\end{split}
\]
Letting $n \to +\infty$ we obtain $\chi_N^+(f,x,v) \le \chi^+(f,x,v)$.

On the other hand, by the differentiability of $f$, given $v \in T_x S$ with $\|v\|=1$ we have
\[
\|Df_x^n \cdot v\| = \lim_{\varepsilon \to 0} \frac{\|f^n(x)-f^n(x+\varepsilon v)\|}{|\varepsilon|}.
\]
For sufficiently small $\varepsilon>0$ such that $x+\varepsilon v \in B_x(\delta,n)$, we have
\[
\begin{split}
\sup_{y \in B_x(\delta,n) \cap L_{x,v}} \Delta(f,n,x,y)
& =\sup_{y \in B_x(\delta,n) \cap L_{x,v}} \frac{\|f^n(x)-f^n(y)\|}{\|x-y\|}\\
& \ge \frac{\|f^n(x)-f^n(x+\varepsilon v)\|}{\|x-(x+\varepsilon v)\|} \\
& = \frac{\|f^n(x)-f^n(x+\varepsilon v)\|}{|\varepsilon|}.
\end{split}
\]
Thus, letting $\varepsilon \to 0$ we obtain
\[
\sup_{y \in B_x(\delta,n) \cap L_{x,v}}  \Delta(f,n,x,y) \ge
\|Df^n_x |_{L_{x,v}}\|.
\]
Therefore
\[
\begin{split}
\chi^+_N(f,x,v)
& = \limsup_{n \to +\infty} \lim_{\delta \to 0} \frac{1}{n}
\log \sup_{y \in B_x(\delta,n) \cap L_{x,v}}  \Delta(f,n,x,y) \\
% & \ge \limsup_{n \to +\infty} \frac{1}{n} \log \|f^n_x |_{L_{x,v}}\| \\
& \ge \chi^+ (f,x,v).
\end{split}
\]
We conclude that $\chi^+(f,x,v)=\chi^+_N(f,x,v)$.

The same reasoning allows us to prove that $\chi^{+}_N(f,x)=\chi^{+}(f,x)$ for all $x \in S$.
\end{proof}

%----------------------------------------------------------------------------%
\subsection{On the invariance of the exponents along orbits}
%----------------------------------------------------------------------------%

The next result shows that, for homeomorphisms, assuming the integrability of some functions, the top new Lyapunov exponents $\chi_N^{+}(f,x)$ are invariant for the orbit of the map.

\begin{theorem}\label{teo:invariancia}
Let $f \in \text{Hom\,}_\lambdaup(S)$. If for all $\delta >0$ sufficiently small and all
$n \in \mathbb{N}$ the function $g_n^\delta: S \to \bar{\mathbb{R}}$ given by
\begin{equation}\label{eq:integravel}
    g_n^\delta(x) = \sup_{y \in B_x(\delta,n)\setminus \{x\}} \log \Delta(f,n,x,y)
\end{equation}
is integrable then, for $\lambdaup$-a.e. $x \in S$, we have
\[
\chi^{+}_N (f,x) =  \lim_{n \to +\infty} \lim_{\delta \to 0} \frac{1}{n} \log \sup_{y \in B_x(\delta,n)\setminus \{x\}}
\Delta(f,n,x,y)
\]
and $\chi^{+}_N(f,x)=\chi^{+}_N(f,f^m (x))$ for all $m \in \ZZ$.
\end{theorem}
\begin{proof}
Let $x \in S$, $\delta>0$ and $n \in \mathbb{N}$. We have, using the fact that $f \in \text{Hom\,}_\lambdaup(S)$,
\begin{eqnarray*}
f\left(B_x(\delta,n+1)\right)&=&f\left(\cap_{k=0}^{n+1} f^{-k}(B(f^k(x),\delta))\right)=\cap_{k=0}^{n+1} f^{-k+1}(B(f^k (x),\delta)) \\
&\subseteq& \cap_{k=1}^{n+1} f^{-k+1}(B(f^k (x),\delta)) =  \cap_{j=0}^{n} f^{-j}(B(f^{j+1} (x),\delta)) \\
&=&\cap_{j=0}^{n} f^{-j}(B(f^{j} (f(x)),\delta)) = \displaystyle B_{f(x)}(\delta,n).
\end{eqnarray*}

Thus we have, for all $m,n \in \NN$,
\[
\begin{split}
g^\delta_{n+m}(x)
& = \sup_{y \in B_x(\delta,m+n)} \log \Delta(f,n,x,y) \\
& = \sup_{y \in B_x(\delta,m+n)} \log
\left[ \frac{\|f^{n+m}(x) - f^{n+m}(y)\|}{\|f^m (x) - f^m (y)\|} \,
\frac{\|f^m (x) - f^m (y)\|}{\|x - y\|} \right] \\
& \le \sup_{y \in B_x(\delta,m+n)} \log
\frac{\|f^{n+m}(x) - f^{n+m}(y)\|}{\|f^m (x) - f^m (y)\|} +
\sup_{y \in B_x(\delta,m+n)}
\log \frac{\|f^m (x) - f^m (y)\|}{\|x - y\|} \\
& \le \sup_{u \in B_{f^m (x)}(\delta,n)} \log
\frac{\|f^n(f^m (x)) - f^n (u)\|}{\|f^m (x) - u\|} +
\sup_{y \in B_x(\delta,m)}
\log \frac{\|f^m (x) - f^m (y)\|}{\|x - y\|} \\
& = \sup_{u \in B_{f^m (x)}(\delta,n)} \log
\Delta(f,n,f^m (x),u) +
\sup_{y \in B_x(\delta,m)}
\log \Delta(f,m,x,y) \\
& = g_n^\delta (f^m (x)) + g^\delta_m (x).
\end{split}
\]
and we conclude that $n \mapsto g^\delta_n(x)$ is subadditive.
By Kingman's subadditive ergodic theorem, using the
integrability of~\eqref{eq:integravel}, we obtain, for $\lambdaup$-a.e.
$x \in S$, the existence of the limit
\[
\lim_{n \to \infty} \frac{1}{n} \sup_{y \in B_x(\delta,n)}
\log \Delta(f,n,x,y)
\]
and the $f$-invariance of the function $x \mapsto g^\delta_n(x)$.
Therefore we have, since $f$ is a homeomorphism,
$g^\delta_n(f^m (x)) = g^\delta_n(x)$ for all $m \in \ZZ$.

Letting $\delta \to 0$ we have
\[
\lim_{\delta \to 0} \frac{1}{n} \log g^\delta_n(f^m (x))
= \lim_{\delta \to 0} \frac{1}{n}  \log g^\delta_n(x),
\]
and letting $n \to +\infty$, we finally obtain
$\chi^+(f,f^m (x))=\chi^+(f,x)$ for all $m \in \ZZ$.
\end{proof}

%----------------------------------------------------------------------------%
\subsection{New Lyapunov exponents and the dimension of the manifold}
%----------------------------------------------------------------------------%

We remark that we loose some features of the Lyapunov exponent when we don't have differentiability. In particular, since we no longer have the linearity given by the derivative, the number of Lyapunov exponents is no longer bounded by the dimension of the manifold. We illustrate this fact with the next example.

\begin{example}
Let $f:\RR^2 \to \RR^2$ be given by
$$f(x,y)
=\begin{cases}
(2x,\, \frac32 x + \frac12 y), & \text{if} \quad x(y-x) > 0 \\
(3x-y,\, 2y), & \text{if} \quad x(y-x) < 0 \\
(3x,\, \frac12 y), & \text{if} \quad xy \le 0 \\
(2x,\, 2y), & \text{if} \quad x=y
\end{cases}.$$
We have $f^n (0,0)=(0,0)$ for all $n \in \NN$. On the other hand, if $h \ne 0$, $f^n(h,h)=(2^n h,2^n h)$, $f^n(0,h)=(0,h / 2^n)$ and $f^n(h,0)=(3^n h, 0)$.

We thus have, if $h \ne 0$,
$$\Delta(f,n,(0,0),(h,h)) = \frac{\|(2^n h,2^n h)\|}{\|(h,h)\|} = 2^n,$$ $$\Delta(f,n,(0,0),(0,h)) = \frac{\|(0,h / 2^n)\|}{\|(0,h)\|} = \frac{1}{2^n}$$ and
$$\Delta(f,n,(0,0),(h,0)) = \frac{\|(3^n h, 0)\|}{\|(h,0)\|} = 3^n.$$

It follows that, if $h \ne 0$, $\chi^+_N (f,(0,0),(h,h)) = \log 2$, $\chi^+_N (f,(0,0),(0,h)) = - \log 2$ and
$\chi^+_N (f,(0,0),(h,0)) = \log 3$.

We emphasize that $f$ is not differentiable in $(0,0)$. Observe also that $f$ is dissipative.
\end{example}

Note that the definition of Lyapunov exponent in~\eqref{eq:LyapunovExp} can be ge\-ne\-ra\-lized in a straightforward manner to arbitrary continuous maps defined in $n$-dimensional manifolds and it is immediate that a co\-rres\-pon\-ding version Theorem~\ref{main0} still holds in that case. It is also easy to recognize that a direct version of Theorem~\ref{teo:invariancia} holds for homeomorphisms in arbitrary dimension. For continuous invertible maps, we can also define a corresponding notion of backward Lyapunov exponent.

%----------------------------------------------------------------------------%
\subsection{Oseledets' theorem and topological realizable sequences}\label{TRS}
%----------------------------------------------------------------------------%

 The next result, due to Oseledets~\cite{O}, is a cornerstone in smooth ergodic theory. For a proof in the two-dimensional context see ~\cite{MP}.

\begin{theorem}\label{Oseledec}

Let $f\in{\text{Diff\,}_{\lambdaup}(S)}$. For $\lambdaup$-a.e. $x\in{S}$
there exists the upper Lyapunov exponent $\chi^{+}(f,x)$
defined by the limit
$\underset{n\rightarrow{+\infty}}{\lim}\frac{1}{n}\log\|Df_x^{n}\|$
which is a non-negative measurable function of $x$. For $\lambdaup$-a.e.
point $x$ with a positive Lyapunov exponent there is a splitting of the tangent space  $T_{x}S=E_{x}^{u}\oplus{E_{x}^{s}}$ (where $\dim(E_x^s)=\dim(E_x^u)=1$) which varies measurably
with $x$ and satisfy the following equalities:
\begin{itemize}
\item If ${v}\in{E_{x}^{u}}\setminus\{\vec{0}\}$, then
$\underset{n\rightarrow{\pm\infty}}{\lim}\frac{1}{n}\log\|Df_{x}^{n}\cdot v\|=\chi^{+}(f,x)$.
\item If ${v}\in{E_{x}^{s}}\setminus\{\vec{0}\}$, then
$\underset{n\rightarrow{\pm\infty}}{\lim}\frac{1}{n}\log\|Df_x^{n}\cdot v\|=-\chi^{+}(f,x)$.
\item If
$\vec{0}\not={v}\notin{E_{x}^{u}},E_{x}^{s}$,
then\\
$\underset{n\rightarrow{+\infty}}{\lim}\frac{1}{n}\log\|Df_x^{n}\cdot v\|=\chi^{+}(f,x)$ and
$\underset{n\rightarrow{-\infty}}{\lim}\frac{1}{n}\log\|Df_{x}^{n}\cdot v\|=-\chi^{+}(f,x)$.
\end{itemize}
\end{theorem}

In order to prove his theorem Bochi used the lack of hyperbolic behavior to blend different expansion rates in the Oseledets splitting which caused a decay on the upper Lyapunov exponent. Actually, the next definition, adapted to the one given by Bochi (\cite[page 1672]{Boc}), is fundamental if one wants to give a direct prove of Theorem~\ref{main} or to generalize it to higher dimensions (cf. \S ~\ref{G}).

\begin{definition}\label{trs}
Let be given $f\in \text{Diff\,}_\omega(S)$, $\epsilon>0$, $0<\kappa<1$ and a non-periodic point $p$. We say that the sequence of
   $L_i\in SL(2,\mathbb{R})$ for $i=0,1,...,n-1$ such that
   $$T_p S\overset{L_{0}}{\longrightarrow}{T_{f(p)}S}\overset{L_{1}}
      {\longrightarrow}{T_{f^2(p)}S}\overset{L_{2}}
      {\longrightarrow}...\overset{L_{n-2}}{\longrightarrow}
      {T_{f^{n-1}(p)}S}\overset{L_{n-1}}{\longrightarrow}{T_{f^{n}(p)}S}$$
is an \textbf{$(\epsilon,\kappa)$-topological realizable sequence of length $n$ at $p$} if  for all $\gamma>0$, there exists ${r}>0$ such that, for any open set $\emptyset\not=U\subseteq{B(p,r)}$, there exists:
\begin{enumerate}
  \item [(a)] A measurable set $K\subseteq{U}$ such that $\omega(K)>(1-\kappa)\omega(U)$,
  \item [(b)] An area-preserving diffeomorphism $g$, $\epsilon$-$C^{0}$-close to $f$, such that:
\end{enumerate}
\begin{enumerate}
   \item [(i)] $f=g$ outside $\cup_{i=0}^{n-1}f^{i}(\overline{U})$ and the iterates $f^{j}(\overline{B(p,r)})$ are two-by-two disjoint for any $j$;
   \item [(ii)] If $q\in{K}$, then $\|Dg_{g^{j}(q)}-L_j\|<\gamma$ for $j\in\{0,1,...,n-1\}$.
\end{enumerate}
\end{definition}

\begin{remark}
Note that our definition of $(\epsilon,\kappa)$-topological realizable sequence of length $n$ at $p$ is with respect to the $C^0$-topology while Bochi's definition is with respect to the $C^1$-topology.
\end{remark}

\begin{remark}
Observe that, for the topological case, it is easy to blend one-dimensional fibers of Oseledets' splitting with small perturbations, whereas for the $C^1$-case the hyperbolicity (Anosov) is a strong obstruction.
\end{remark}

Instead of using a direct proof of Theorem~\ref{main} we will use the statement of Bochi-Ma\~n\'e dichotomy (\cite{Boc}). This is the content of the next section.

%----------------------------------------------------------------------------%
%----------------------------------------------------------------------------%
\section{Proof of Theorem~\ref{main}} \label{PL}
%----------------------------------------------------------------------------%
%----------------------------------------------------------------------------%

The next result, due to Arbieto and Matheus (\cite[Theorem 3.6]{AM}) allows us to change an area-preserving diffeomorphism by its derivative in a small neighborhood of a periodic point and then extend the construction (conservatively and in a $C^1$-smooth fashion) to the whole $S$. It is called \emph{weak} because we need to start with an area-preserving diffeomorphism with some regularity (of class $C^2$ at least).

\begin{theorem}(weak Pasting Lemma~\cite[Theorem 3.6]{AM})\label{WPL}
If $f\in\text{Diff\,}_{\lambdaup}(S)$ is of class $C^2$ and $x\in S$, then for any $\alpha\in]0,1[$ and $\epsilon>0$, there exists an $\epsilon$-$C^1$-perturbation $g\in\text{Diff\,}_{\lambdaup}(S)$ (which
is a $C^{1+\alpha}$ diffeomorphism) of $f$ such that, for some small neighborhoods $U\supset V$ of $x$ we get
\begin{enumerate}
\item $g|_{U^c} = f$, where $U^c$ denotes the complementary set of $U$ and
\item in local charts we have $g|_V = Df_x$.
\end{enumerate}
\end{theorem}

\begin{remark}
Observe that in Theorem~\ref{WPL} we required that, in the small neighborhood $V$ the perturbation is equal to the tangent map, and thus we get $C^1$-closeness. If we choose any linear map in $V$ the conservative construction also holds but, of course, we loose the  $C^1$-closeness and instead we obtain $C^0$-closeness as long as $V$ is very small.
\end{remark}

The next simple result will be very useful to prove Theorem~\ref{main}. The subtle step is to be careful to perform the perturbations without leaving the area-preserving setting.

\begin{lemma}\label{NA}
Given any area-preserving Anosov diffeomorphism $f\colon S\rightarrow S$ and $\epsilon>0$, there exists a non-Anosov area-preserving diffeomorphism $g$ such that $g$ is $\epsilon$-$C^0$-close to $f$.
\end{lemma}

\begin{proof}
Take the fixed point for $f$ and denote it by $p\in S$. Pick, using the smoothing result by Zehnder (\cite{Z}), $f_1\in \text{Diff\,}_{\lambdaup}(S)$ of class $C^\infty$ and let $p_1$ be the fixed point for $f_1$ given by the analytic continuation of $p$. Choose a small neighborhood $\mathcal{V}_1$ of $p_1$. Now, using Theorem~\ref{WPL}, it is easy to built a diffeomorphism $f_2$ such that:
\begin{enumerate}
\item $f_2$ is area-preserving,
\item $f_2=id$ ($id=$Identity) inside a small compact contained in $\mathcal{V}_1$,
\item $f_2=f_1$ outside a very small open set $\mathcal{U}\supset \mathcal{V}_1$ and
\item $f_2$ is $\epsilon$-$C^0$-close to $f$ (although $C^1$-far from $f$).
\end{enumerate}
Take $g=f_2$, $g$ cannot be an Anosov diffeomorphism and the lemma is proved.

\end{proof}

\begin{proof}(of Theorem~\ref{main})
We will prove that given $h\in\text{Hom\,}_\lambdaup (S)$ and $\epsilon>0$, there exists $f\in\text{Hom\,}_\lambdaup (S)$ such that $\lambdaup$-a.e. point $x\in S$ satisfy $\chi^+_N(f,x)=0$ and $f$ is $\epsilon$-$C^0$-close to $h$. First,  by using (\cite{Oh,S}), consider $f_1\in \text{Diff\,}_{\lambdaup}(S)$ such that $f_1$ is $\frac{\epsilon}{3}$-$C^0$-close to $h$.

We assume that $f_1$ is far from the set of area-preserving Anosov diffeomorphisms, otherwise, by using Lemma~\ref{NA}, we can guarantee the existence of $f_2\in \text{Diff\,}_{\lambdaup}(S)$ such that $f_2$ is $\frac{\epsilon}{3}$-$C^0$-close to $f_1$ and $f_2$ is $C^0$-far (thus $C^1$-far) from the Anosov area-preserving diffeomorphisms. Then, using Bochi-Ma\~n\'e theorem (\cite{Boc}) we obtain $f_3\in \text{Diff\,}_{\lambdaup}(S)$ such that $f_3$ is $\frac{\epsilon}{3}$-$C^0$-close to $f_2$ and $f_3$ has $\chi^+(f_3,x)=0$ for $\lambdaup$-a.e. $x\in S$. Finally, since $f_1$ is differentiable, we use Theorem~\ref{main0} and we obtain that $\chi^+(f_3,x)=\chi_N^+(f_3,x)=0$ for $\lambdaup$-a.e. $x\in S$ and the theorem is proved by taking $g=f_3$.

\end{proof}

\begin{remark}
Of course that Theorem~\ref{main} also holds if we consider the standard Lyapunov exponents instead of our top \emph{new} Lyapunov exponents. However, it is misleading to state something about an object (the standard Lyapunov exponents) among a setting $\text{Hom\,}_\lambdaup (S)$ where they are not even defined.
\end{remark}

\medskip

We finish this section with the following question.

\medskip

\begin{question}
Does Theorem~\ref{main} also holds if we switch our top new exponents by the ones in ~\cite{BS}?
\end{question}

%----------------------------------------------------------------------------%
\section{Proof of Theorem~\ref{main2}}\label{Beguin}
%----------------------------------------------------------------------------%

Recall the function

$$
\begin{array}{cccc}
\Lambda\colon & \text{Hom\,}_\lambdaup (S) & \longrightarrow & [0,+\infty[ \\
& h & \longmapsto & \int_S \chi_N^{+}(h,x)d\lambdaup(x),
\end{array}
$$
where $\text{Hom\,}_\lambdaup (S)$ is endowed with the $C^0$-topology and $\chi_N^{+}$ is the top \emph{new} Lyapunov exponent (as in \S\ref{NLE}). We know that $h\in \text{Hom\,}_\lambdaup (S)$ has zero \emph{new} Lyapunov exponents for $\lambdaup$-a.e. point in $S$, if and only if $\Lambda(h)=0$.

In  \cite[Proposition 2.1]{Boc} it was proved that
$$
\begin{array}{cccc}
\hat{\Lambda}\colon & \text{Diff\,}_{\lambdaup}(S) & \longrightarrow & [0,+\infty[ \\
& f & \longmapsto & \int_S \chi^{+}(f,x)d\lambdaup(x),
\end{array}
$$
where $\text{Diff\,}_{\lambdaup}(S)$ is endowed with the $C^1$-topology is an upper-semicontinuous function, hence the points of continuity of $\hat{\Lambda}$ forms a residual subset of $\text{Diff\,}_{\lambdaup}(S)$ (see~\cite{F}) and this was the key step to obtain Bochi-Ma\~n\'e theorem. Unfortunately, in the topological context, such a good property of upper-semicontinuity does not holds, and thus, is unattainable to go on with the strategy in \cite{Boc}. We would like to thank Fran\c{c}ois B\'eguin for suggesting us the construction in the proof of the next result.

\begin{maintheorem}
  \label{main2}
The function $\Lambda$ is not upper-semicontinuous.
\end{maintheorem}

\begin{proof}

Let $S$ be a compact surface. We will prove that $\Lambda$ is not upper-semicontinuous at $id\colon S\rightarrow S$.

Let $\mathcal{D}\subset S$ be a disc of radius equal to $1$. By Katok's construction (see~\cite{Ka}) there exists a $C^\infty$ diffeomorphism $\kappa\colon \mathcal{D}\rightarrow \mathcal{D}$ such that $\kappa$ is an ergodic area-preserving diffeomorphism in $\mathcal{D}$ and has non-zero Lyapunov exponents for $\lambdaup$-a.e. $x\in \mathcal{D}$ say equal to $\chi>0$ and $-\chi$.

On a surface $S$, consider a family of $n^2$ pairwise disjoint discs $\{\mathcal{D}_i\}_{i=1}^{n^2}$ such that each disc has radius $\frac{k}{10n}$ (for some fixed and adequate $k>0$ depending on $S$). We get that the area of each disc is equal to $\frac{\pi k^2}{100n^2}$. Now, for each $i\in\{1,...,n^2\}$ we  consider the linear homothety $\ell_i\colon \mathcal{D}\rightarrow\mathcal{D}_i$ deforming $\mathcal{D}$ into $\mathcal{D}_i$. Finally, construct an area-preserving homeomorphism $g_n\colon S\rightarrow S$ such that:

\begin{enumerate}
\item $g_n=id$ outside $\bigcup_{i=1}^{n^2}\mathcal{D}_i$ and
\item $g_n=\ell_i\circ \kappa\circ\ell_i^{-1}$ on each $\mathcal{D}_i$.
\end{enumerate}

We observe that $g_n$ is $\frac{k}{10n}$-$C^0$-close to $id$. Now,

\begin{eqnarray*}
\Lambda(g_n)&=&\int_S \chi^+_N(g_n,x)d\lambdaup(x)=\int_{\bigcup_{i=1}^{n^2}\mathcal{D}_i} \chi^+_N(g_n,x)d\lambdaup(x)\\&=&\int_{\bigcup_{i=1}^{n^2}\mathcal{D}_i} \chi^+(g_n,x)d\lambdaup(x)=\int_{\bigcup_{i=1}^{n^2}\mathcal{D}_i} \chi d\lambdaup(x)\\&=&\chi\lambdaup\left(\cup_{i=1}^{n^2}\mathcal{D}_i\right)=\chi\sum_{i=1}^{n^2}\lambdaup(\mathcal{D}_i)=\chi\frac{\pi\, k}{100}
\end{eqnarray*}

Finally, we observe that $\Lambda(g_n)= \chi\frac{\pi\, k}{100}$ (for any $n$), but on the other hand $g_n\rightarrow id$ (in the $C^0$-sense) as $n\rightarrow\infty$, thus there is ``jump'' for $\Lambda$ at $id$ and so $\Lambda$ cannot be upper-semicontinuous.

\end{proof}

%----------------------------------------------------------------------------%
\subsection*{Acknowledgments}
 %----------------------------------------------------------------------------%
MB was partially supported by FCT - Funda\c{c}\~ao para a Ci\^encia e a Tecnologia through CMUP (SFRH/BPD/20890/2004) and CS was partially supported by FCT through CMUBI (FCT/POCI2010/FEDER). CS would like to thank Lu\'is Barreira for fruitful conversations aiming the definition of the exponents for continuous transformations.

%----------------------------------------------------------------------------%
%      \bibliographystyle{abbrv}
%      \bibliography{../bibliobase/bibliography}

\def\cprime{$'$}

\end{document}